\documentclass[reqno]{amsart}
\usepackage{amsmath,amssymb,amsfonts,latexsym,txfonts,pxfonts,wasysym}
\usepackage{amssymb,longtable}
\usepackage{hyperref}
\usepackage{lscape,graphicx}
\allowdisplaybreaks
\begin{document}
\title[ On Certain Hypergeometric Identities ]
{On Certain Hypergeometric Identities Deducible by using Beta Integral Method}

\author[A. K. Ibrahim, M. A. Rakha, A. K. Rathie]
{Adel K. Ibrahim, Medhat A. Rakha, Arjun K. Rathie}  % in alphabetical order

\address{Adel K. Ibrahim \newline
 Mathematics Department, 
 College of Science,
 Jazan University University, Jazan, Saudi Arabia}  
 \email{dradlkhalil@yahoo.com}

\address{Medhat A. Rakha \newline
 Department of Mathematics and Statistics,
 College of Science,
 Sultan Qaboos University,
 P.O.Box 36 - Al-Khoud 123, 
 Muscat - Sultanate of Oman}
\email{medhat@squ.edu.om}

\address{Arjun K. Rathie \newline
 Department of Mathematics, 
 School of Mathematical and Physical Sciences,
 Central University of Kerala, Riverside Transit Campus,
 Padennakkad P.O. Nileshwar, Kasaragod - 671 328, Kerala - INDIA}
 \email{akrathie@cukerala.edu.in}


\begin{abstract}
The aim of this research paper is to demonstrate how one can obtain eleven new
and interesting hypergeometric identities (in the form of a single result)
from the old ones by mainly applying the well known beta integral method which
was used successfully and systematically by Krattenthaler and Rao in their
well known, very interesting research papers.

The results are derived with the help of generalization of a quadratic
transformation formula due to Kummer very recently obtained by Kim, et al. .
Several identities including one obtained earlier by Krattenthaler and Rao
follow special cases of our main findings.

The results established in this paper are simple, interesting, easily
established and may be potentially useful.

\end{abstract}

\maketitle
\numberwithin{equation}{section}
\newtheorem{theorem}{Theorem}[section]
\newtheorem{lemma}[theorem]{Lemma}
\newtheorem{proposition}[theorem]{Proposition}
\newtheorem{corollary}[theorem]{Corollary}
\newtheorem*{remark}{Remark}

\section{Introduction and Preliminaries}

The generalized hypergeometric series $_{p}F_{q}$ is defined by \cite{2,7}
\begin{equation}
_{p}F_{q}\left[
\begin{array}
[l]{llll}
\alpha_{1}, & \ldots, & \alpha_{p}; & \\
&  &  & z\\
\beta_{1}, & \ldots, & \beta_{q}; &
\end{array}
\right]  =
{\displaystyle\sum\limits_{n=0}^{\infty}}
\frac{\left(  \alpha_{1}\right)  _{n}\ldots\left(  \alpha_{p}\right)  _{n}%
}{\left(  \beta_{1}\right)  _{n}\ldots\left(  \beta_{q}\right)  _{n}}%
\frac{z^{n}}{n!} \label{1.1}%
\end{equation}
where $\left(  a\right)  _{n}$ is the Pochhammer symbol (or the shifted or
raised factorial, since $\left(  1\right)  _{n}=n!$) defined (for
$a\in\mathbb{C}$) by%
\begin{align}
\left(  a\right)  _{n}  &  =\left\{
\begin{array}
[l]{l}%
1\\
a(a+1)\ldots(a+n-1)
\end{array}
\right.
\begin{array}
[c]{c}
n=0\\
n\in\mathbb{N}=\{1,2,\ldots\}
\end{array}
\nonumber\\
&  =\frac{\Gamma\left(  a+n\right)  }{\Gamma\left(  a\right)  },\,a\in
\mathbb{C}\smallsetminus\mathbb{Z}_{0}^{-} \label{1.2}%
\end{align}
and $\mathbb{Z}_{0}^{-}$ denotes the set of nonpositive integers, $\mathbb{C}
$ the set of complex numbers, and $\Gamma\left(  a\right)  $ is the familiar
Gamma function. Here $p$ and $q$ are positive integers or zero (interpreting
an empty product as unity), and we assume for simplicity that the variable $z
$, the numerator parameters $\alpha_{1},\ldots,\alpha_{p}$ and the denominator
parameters $\beta_{1},\ldots,\beta_{q}$ take on complex values, provided that
no zeros appear in the denominator of (\ref{1.1}), that is
\begin{equation}
\beta_{j}\in\mathbb{C}\smallsetminus\mathbb{Z}_{0}^{-};\qquad j=1,\ldots,q.
\label{1.3}%
\end{equation}

For the detailed conditions of the convergence of the series (\ref{1.1}), we
refer to \cite{7}. It is not out of place to mention here that if one of the
numerator parameters, say $a_{j}$ is a negative integer, then the series
(\ref{1.1}) reduces to a polynomial in $z$ of degree $-a_{j}$.

It is interesting to mention here that whenever a generalized hypergeometric
functions reduces to gamma function, the results are very important from the application point of view.
Thus the classical summation theorems such as those of Gauss, Gauss second,
Kummer and Bailey for the series $_{2}F_{1}$; Watson, Dixon, Whipple and
Saalsch\"{u}tz for the series $_{3}F_{2}$ and others play an important role in the
theory of hypergeometric and generalized hypergeometric series.

In a very interesting, popular and useful research paper, Bailey \cite{1}, by
employing the above mentioned classical summation theorems, obtained a large
number of results involving products of generalized hypergeometric series as
well as quadratic and cubic transformations. Several other results were also
given by Gauss and Kummer.

Evidently, if the product of two generalized hypergeometric series can be
expressed as another generalized hypergeometric series with argument $x$, the
coefficients of $x^{n}$ in the product must be expressible in terms of gamma functions.

In our present investigation, we are interested in the following quadratic
transformation due to Kummer \cite{6}%
\begin{equation}
(1-x)^{-2a}\,_{2}F_{1}\left[
\begin{array}
[l]{lll}
2a, & b; & \\
&  & -\frac{2x}{1-x}\\
2b; &  &
\end{array}
\right]  =\,_{2}F_{1}\left[
\begin{array}
[l]{lll}
a, & a+\frac{1}{2}; & \\
&  & x^{2}\\
b+\frac{1}{2}; &  &
\end{array}
\right]  . \label{1.4}%
\end{equation}

This result was independently rediscovered by Ramanujan \cite[Entry 2, p.49]{3}

By assuming $a$ to be a nonpositive integer and employing the so-called Beta
integral method, recently Krattenthaler and Rao \cite[Eq. (3.4), p. 164]{5}
obtained the following interesting identity,%
\begin{align}
&  _{4}F_{3}\left[
\begin{array}
[l]{lllll}
a, & a+\frac{1}{2}, & \frac{1}{2}d, & \frac{1}{2}d+\frac{1}{2}; & \\
&  &  &  & 1\\
b+\frac{1}{2}, & \frac{1}{2}e, & \frac{1}{2}e+\frac{1}{2}; &  &
\end{array}
\right] \nonumber\\
&  =\frac{\Gamma\left(  e\right)  \Gamma\left(  e-2a-d\right)  }{\Gamma\left(
e-2a\right)  \Gamma\left(  e-d\right)  }\,_{3}F_{2}\left[
\begin{array}
[l]{llll}
2a, & b, & d; & \\
&  &  & 2\\
2b, & 1+2a+d-e; &  &
\end{array}
\right]  \label{1.5}%
\end{align}
provided $a$ or $d$ is a nonpositive integer.

Very recently, Kim, et al. \cite{4} have obtained the following generalization
of the Kummer quadratic transformation formula (\ref{1.4}) in the form
\begin{align}
&  (1-x)^{-2a}\,_{2}F_{1}\left[
\begin{array}
[l]{lll}
2a, & b; & \\
&  & -\frac{2x}{1-x}\\
2b+j; &  &
\end{array}
\right] \nonumber\\
&  =\frac{\Gamma\left(  b\right)  \Gamma\left(  1-b\right)  }{\Gamma\left(
b+\frac{1}{2}j+\frac{1}{2}\left\vert j\right\vert \right)  \Gamma\left(
1-b-\left[  \frac{j+1}{2}\right]  \right)  }\nonumber\\
&  \cdot{\displaystyle\sum\limits_{n=0}^{\infty}}
A_{j}\frac{x^{2n}}{n!}\frac{\left(  a\right)  _{n}\left(  a+\frac{1}%
{2}\right)  _{n}\left(  b+\left[  \frac{j+1}{2}\right]  \right)  _{n}}{\left(
b+\frac{1}{2}j\right)  _{n}\left(  b+\frac{1}{2}j+\frac{1}{2}\right)  _{n}%
}\nonumber\\
&  +\frac{2a}{(2b+j)}\frac{\Gamma\left(  -b\right)  \Gamma\left(  1+b\right)
}{\Gamma\left(  -b-\left[  \frac{j}{2}\right]  \right)  \Gamma\left(
b+\frac{1}{2}j+\frac{1}{2}\left\vert j\right\vert \right)  }\nonumber\\
&  \cdot{\displaystyle\sum\limits_{n=0}^{\infty}}
B_{j}\frac{x^{2n+1}}{n!}\frac{\left(  a+\frac{1}{2}\right)  _{n}\left(
a+1\right)  _{n}\left(  b+1+\left[  \frac{j}{2}\right]  \right)  _{n}}{\left(
b+\frac{1}{2}j+\frac{1}{2}\right)  _{n}\left(  b+\frac{1}{2}j+1\right)  _{n}}
\label{1.6}
\end{align}
for $j=0,\pm1,\pm2,\pm3,\pm4,\pm5.$

Here, as usual, $[x]$ denotes the greatest integer less than or equal to $x$
and its modulus is denoted by $\left\vert x\right\vert$. The coefficients,
$A_{j}$ and $B_{j}$ are given in the following table.

\bigskip
\begin{center}
\begin{tabular}
[c]{|l|l|l|}\hline
$\mathbf{j}$ & $\mathbf{A}_{j}$ & $\mathbf{B}_{j}$\\\hline
$5$ & $%
\begin{array}
[l]{l}
-4(1-b-2n)^{2}\\
+2(1-b)(1-b-2n)+(1-b)^{2}\\
+22(1-b-2n)+13b-33
\end{array}
$ & $
\begin{array}
[l]{l}
4(b+2n)^{2}-2(1-b)(b+2n)\\
-(1-b)^{2}+34(b+2n)+b+61
\end{array}
$\\\hline
$4$ & $%
\begin{array}
[l]{l}
2(b+1+2n)(b+3+2n)\\
-b(b+3)
\end{array}
$ & $4(b+3+2n)$\\\hline
$3$ & $b+2+4n$ & $-(3b+6+4n)$\\\hline
$2$ & $-(b+1+2n)$ & $-2$\\\hline
$1$ & $-1$ & $1$\\\hline
$0$ & $1$ & $0$\\\hline
$-1$ & $1$ & $1$\\\hline
$-2$ & $1-b-2n$ & $2$\\\hline
$-3$ & $1-b-4n$ & $3-3b-4n$\\\hline
$-4$ & $%
\begin{array}
[l]{l}
2(1-b-2n)(3-b-2n)\\
-(1-b)(4-b)
\end{array}
$ & $4(1-b-2n)$\\\hline
$-5$ & $%
\begin{array}
[l]{l}
4(1-b-2n)^{2}\\
-2(1-b)(1-b-2n)-(1-b)^{2}\\
+8(1-b-2n)+7b-7
\end{array}
$ & $
\begin{array}
[l]{l}
4(b+2n)^{2}-2(1-b)(b+2n)\\
-(1-b)^{2}-16(b+2n)+b-1
\end{array}
$\\\hline
%\caption{Table of the coefficients, $A_{j}$ and $B_{j}$}\label{Table-1}
\end{tabular}
\end{center}
\bigskip

Here, in this paper, we show how one can easily obtain eleven interesting
hypergeometric identities including the Krattenthaler-Rao result (\ref{1.5})
in the form of a single unified result by employing the beta integral method
developed by Krattenthaler and Rao \cite{5}. The results are derived with the
help of the generalization (\ref{1.6}) of the Kummer's formula (\ref{1.4}).
Several interesting special cases of our main result including (\ref{1.5}) are
also explicitly demonstarted.

The results presented in this paper are simple, interesting, easily
established and (potentially) useful.

\section{Main Result}

Our eleven main identities are given here in the form of a single unified
result asserted in the following theorem.

\begin{theorem}
For $a$ or $d$ to be a nonpositive integer, the following generalization of
Krattenthaler-Rao formula (\ref{1.5}) holds true%
\begin{align}
&  \frac{\Gamma\left(  e\right)  \Gamma\left(  e-2a-d\right)  }{\Gamma\left(
e-2a\right)  \Gamma\left(  e-d\right)  }\,_{3}F_{2}\left[
\begin{array}
[l]{llll}
2a, & b, & d; & \\
&  &  & 1\\
2b+j, & 1+2a+d-e; &  &
\end{array}
\right] \nonumber\\
&  =\frac{\Gamma\left(  b\right)  \Gamma\left(  1-b\right)  }{\Gamma\left(
b+\frac{1}{2}j+\frac{1}{2}\left\vert j\right\vert \right)  \Gamma\left(
1-b-\left[  \frac{j+1}{2}\right]  \right)  }\nonumber\\
&  \cdot{\displaystyle\sum\limits_{n=0}^{\infty}}
\frac{A_{j}}{n!}\frac{\left(  a\right)  _{n}\left(  a+\frac{1}{2}\right)
_{n}\left(  b+\left[  \frac{j+1}{2}\right]  \right)  _{n}}{\left(  b+\frac
{1}{2}j\right)  _{n}\left(  b+\frac{1}{2}j+\frac{1}{2}\right)  _{n}}%
\frac{\left(  \frac{1}{2}d\right)  _{n}\left(  \frac{1}{2}d+\frac{1}%
{2}\right)  _{n}}{\left(  \frac{1}{2}e\right)  _{n}\left(  \frac{1}{2}%
e+\frac{1}{2}\right)  _{n}}\nonumber\\
&  +\frac{2a}{(2b+j)}\left(\frac{d}{e}\right)\frac{\Gamma\left(  -b\right)  \Gamma\left(  1+b\right)
}{\Gamma\left(  b+\frac{1}{2}j+\frac{1}{2}\left\vert j\right\vert \right)
\Gamma\left(  -b-\left[  \frac{j}{2}\right]  \right)} \nonumber\\
&  \cdot{\displaystyle\sum\limits_{n=0}^{\infty}}
\frac{B_{j}}{n!}\frac{\left(  a+\frac{1}{2}\right)  _{n}\left(  a+1\right)
_{n}\left(  b+1+\left[  \frac{j}{2}\right]  \right)  _{n}}{\left(  b+\frac
{1}{2}j+\frac{1}{2}\right)  _{n}\left(  b+\frac{1}{2}j+1\right)  _{n}}%
\frac{\left(  \frac{1}{2}d+\frac{1}{2}\right)  _{n}\left(  \frac{1}%
{2}d+1\right)  _{n}}{\left(  \frac{1}{2}e+\frac{1}{2}\right)  _{n}\left(
\frac{1}{2}e+1\right)  _{n}} \label{2.1}%
\end{align}
for $j=0,\pm1,\pm2,\pm3,\pm4,\pm5.$ The coefficients $A_{j}$ and $B_{j}$ are same as given in the table.
\end{theorem}

\begin{proof}
Let us first assume that $a$ be a non-positive integer. Multiply the left-hand
side of (\ref{1.6}) by $x^{d-1}(1-x)^{e-d-1}$ and integrating the resulting
equation with respect to $x$ from $0$ to $1$, expressing the involved
$_{2}F_{1}$ as series and changing the order of integration and summation,
which is easily seen to be justified due to the uniform convergence of the
involved series, we have%
\[
L.H.S.={\displaystyle\sum\limits_{n=0}^{\infty}}
\frac{\left(2a\right)_{n}\left(b\right)_{n}\left(-2\right)^{n}}{\left(2b+j\right)_{n}n!}
{\displaystyle\int\limits_{0}^{1}} x^{d+n-1}(1-x)^{e-d-2a-n-1}dx.
\]
Evaluating the beta-integral and using the identity
\[
\Gamma\left(\alpha-n\right)=\frac{(-1)^{n}\Gamma\left(n\right)}{\left(1-\alpha\right)_{n}}
\]
we have, after some algebra
\[
L.H.S.=\frac{\Gamma\left(d\right)\Gamma\left(e-d-2a\right)}{\Gamma\left(e-2a\right)}{\displaystyle\sum\limits_{n=0}^{\infty}}
\frac{\left(  2a\right)  _{n}\left(  b\right)  _{n}\left(  d\right)  _{n}
}{\left(  2b+j\right)  _{n}\left(  1+2a+d-e\right)  _{n}}\frac{2^{n}}{n!}
\]
summing up the series, we have
\begin{equation}
L.H.S.=\frac{\Gamma\left(  d\right)  \Gamma\left(  e-d-2a\right)  }{\Gamma\left(  e-2a\right)  }\,_{3}F_{2}\left[
\begin{array}
[l]{llll}
2a, & b, & d; & \\
&  &  & 2\\
2b+j, & 1+2a+d-e; &  &
\end{array}
\right]  . \label{2.2}%
\end{equation}
Now, multiply the right-hand side of (1.6) by $x^{d-1}(1-x)^{e-d-1}$ and
integrating with respect to $x$ from $0$ to $1$, we have after some
simplification%
\begin{align*}
R.H.S.  &  =\frac{\Gamma\left(b\right)\Gamma\left(1-b\right)}
{\Gamma\left(  b+\frac{1}{2}j+\frac{1}{2}\left\vert j\right\vert \right)
\Gamma\left(  1-b-\left[  \frac{j+1}{2}\right]  \right)  }\\
&  \cdot{\displaystyle\sum\limits_{n=0}^{\infty}}
\frac{A_{j}}{n!}\frac{\left(  a\right)  _{n}\left(  a+\frac{1}{2}\right)
_{n}\left(  b+\left[  \frac{j+1}{2}\right]  \right)  _{n}}{\left(  b+\frac
{1}{2}j\right)  _{n}\left(  b+\frac{1}{2}j+\frac{1}{2}\right)  _{n}}
{\displaystyle\int\limits_{0}^{1}}
x^{d+2n-1}(1-x)^{e-d-1}dx\\
&  +\frac{2a}{(2b+j)}\frac{\Gamma\left(  -b\right)  \Gamma\left(  1+b\right)
}{\Gamma\left(  b+\frac{1}{2}j+\frac{1}{2}\left\vert j\right\vert \right)
\Gamma\left(  -b-\left[  \frac{j}{2}\right]  \right)  }\\
&  \cdot{\displaystyle\sum\limits_{n=0}^{\infty}}
\frac{B_{j}}{n!}\frac{\left(  a+\frac{1}{2}\right)  _{n}\left(  a+1\right)
_{n}\left(  b+1+\left[  \frac{j}{2}\right]  \right)  _{n}}{\left(  b+\frac
{1}{2}j+\frac{1}{2}\right)  _{n}\left(  b+\frac{1}{2}j+1\right) _{n}}
{\displaystyle\int\limits_{0}^{1}}x^{d+2n}(1-x)^{e-d-1}dx
\end{align*}
Evaluating the be\bigskip ta integrals and using the Legendre's duplication
formula%
\[
\Gamma\left(2z\right)=\frac{2^{2z-1}\Gamma\left(z\right)\Gamma\left(z+\frac{1}{2}\right)  }{\sqrt{\pi}},
\]
we have after some simplification%
\begin{align}
R.H.S.  &  =\frac{\Gamma\left(  d\right)  \Gamma\left(  e-d\right)  }
{\Gamma\left(  e\right)  }\left\{  \frac{\Gamma\left(  b\right)  \Gamma\left(
1-b\right)  }{\Gamma\left(  b+\frac{1}{2}j+\frac{1}{2}\left\vert j\right\vert
\right)  \Gamma\left(  1-b-\left[  \frac{j+1}{2}\right]  \right)  }\right.
\nonumber\\
&  \cdot{\displaystyle\sum\limits_{n=0}^{\infty}}
\frac{A_{j}}{n!}\frac{\left(  a\right)  _{n}\left(  a+\frac{1}{2}\right)
_{n}\left(  b+\left[  \frac{j+1}{2}\right]  \right)  _{n}}{\left(  b+\frac
{1}{2}j\right)  _{n}\left(  b+\frac{1}{2}j+\frac{1}{2}\right)  _{n}}%
\frac{\left(  \frac{1}{2}d\right)  _{n}\left(  \frac{1}{2}d+\frac{1}%
{2}\right)  _{n}}{\left(  \frac{1}{2}e\right)  _{n}\left(  \frac{1}{2}%
e+\frac{1}{2}\right)  _{n}}\nonumber\\
&  +\frac{2a}{(2b+j)} \left(\frac{d}{e}\right)\frac{\Gamma\left(  -b\right)  \Gamma\left(  1+b\right)
}{\Gamma\left(  b+\frac{1}{2}j+\frac{1}{2}\left\vert j\right\vert \right)
\Gamma\left(  -b-\left[  \frac{j}{2}\right]  \right)  }\nonumber\\
&  \left.  \cdot{\displaystyle\sum\limits_{n=0}^{\infty}}
\frac{B_{j}}{n!}\frac{\left(  a+\frac{1}{2}\right)  _{n}\left(  a+1\right)
_{n}\left(  b+1+\left[  \frac{j}{2}\right]  \right)  _{n}}{\left(  b+\frac
{1}{2}j+\frac{1}{2}\right)  _{n}\left(  b+\frac{1}{2}j+1\right)  _{n}}%
\frac{\left(  \frac{1}{2}d+\frac{1}{2}\right)  _{n}\left(  \frac{1}%
{2}d+1\right)  _{n}}{\left(  \frac{1}{2}e+\frac{1}{2}\right)  _{n}\left(
\frac{1}{2}e+1\right)  _{n}}\right\}  \label{2.3}
\end{align}
Finally, equating (\ref{2.3}) and (\ref{2.3}), we get the desired result (\ref{2.1}). This completes the proof of (\ref{2.1}).
\end{proof}

\section{Special Cases}

Here we shall consider some of the very interesting special cases of our
main result (\ref{2.1}). Each of the following formulas hold true provided $a$
or $d$ must be a nonpositive integer.

\begin{corollary}%
\begin{align}
&  \frac{\Gamma\left(  e\right)  \Gamma\left(  e-2a-d\right)  }{\Gamma\left(
e-2a\right)  \Gamma\left(  e-d\right)  }\,_{3}F_{2}\left[
\begin{array}
[l]{llll}%
2a, & b, & d; & \\
&  &  & 2\\
2b, & 1+2a+d-e; &  &
\end{array}
\right] \nonumber\\
&  =\,_{4}F_{3}\left[
\begin{array}
[l]{lllll}%
a, & a+\frac{1}{2}, & \frac{1}{2}d, & \frac{1}{2}d+\frac{1}{2}; & \\
&  &  &  & 1\\
b+\frac{1}{2}, & \frac{1}{2}e, & \frac{1}{2}e+\frac{1}{2}; &  &
\end{array}
\right]  \label{3.1}%
\end{align}

\begin{proof}
Setting $j=0$ in (\ref{2.1}) and simplifying the resulting identity, we are led to the formula (\ref{3.1}).
\end{proof}
\end{corollary}

\begin{corollary}%
\begin{align}
&  \frac{\Gamma\left(  e\right)  \Gamma\left(  e-2a-d\right)  }{\Gamma\left(
e-2a\right)  \Gamma\left(  e-d\right)  }\,_{3}F_{2}\left[
\begin{array}
[l]{llll}
2a, & b, & d; & \\
&  &  & 2\\
2b+1, & 1+2a+d-e; &  &
\end{array}
\right] \nonumber\\
&  =\,_{4}F_{3}\left[
\begin{array}
[l]{lllll}
a, & a+\frac{1}{2}, & \frac{1}{2}d, & \frac{1}{2}d+\frac{1}{2}; & \\
&  &  &  & 1\\
b+\frac{1}{2}, & \frac{1}{2}e, & \frac{1}{2}e+\frac{1}{2}; &  &
\end{array}
\right] \nonumber\\
&  +\frac{2ad}{e(2b+1)}\,_{4}F_{3}\left[
\begin{array}
[l]{lllll}
a+\frac{1}{2}, & a+1, & \frac{1}{2}d+\frac{1}{2}, & \frac{1}{2}d+1; & \\
&  &  &  & 1\\
b+\frac{3}{2}, & \frac{1}{2}e+\frac{1}{2}, & \frac{1}{2}e+1; &  &
\end{array}
\right]  \label{3.2}%
\end{align}

\begin{proof}
Setting $j=1$ in (\ref{2.1}) and simplifying the resulting identity, we are led to the formula (\ref{3.2}).
\end{proof}
\end{corollary}

\begin{corollary}%
\begin{align}
&  \frac{\Gamma\left(  e\right)  \Gamma\left(  e-2a-d\right)  }{\Gamma\left(e-2a\right)  \Gamma\left(  e-d\right)  }\,_{3}F_{2}\left[
\begin{array}
[l]{llll}%
2a, & b, & d; & \\
&  &  & 2\\
2b-1, & 1+2a+d-e; &  &
\end{array}
\right] \nonumber\\
&  =\,_{4}F_{3}\left[
\begin{array}
[l]{lllll}
a, & a+\frac{1}{2}, & \frac{1}{2}d, & \frac{1}{2}d+\frac{1}{2}; & \\
&  &  &  & 1\\
b-\frac{1}{2}, & \frac{1}{2}e, & \frac{1}{2}e+\frac{1}{2}; &  &
\end{array}
\right] \nonumber\\
&  -\frac{2ad}{e(2b-1)}\,_{4}F_{3}\left[
\begin{array}
[l]{lllll}
a+\frac{1}{2}, & a+1, & \frac{1}{2}d+\frac{1}{2}, & \frac{1}{2}d+1; & \\
&  &  &  & 1\\
b+\frac{1}{2}, & \frac{1}{2}e+\frac{1}{2}, & \frac{1}{2}e+1; &  &
\end{array}
\right]  \label{3.3}%
\end{align}

\begin{proof}
Setting $j=-1$ in (\ref{2.1}) and simplifying the resulting identity, we are led to the formula (\ref{3.3}).
\end{proof}
\end{corollary}

\begin{corollary}%
\begin{align}
&  \frac{\Gamma\left(  e\right)  \Gamma\left(  e-2a-d\right)  }{\Gamma\left(e-2a\right)  \Gamma\left(  e-d\right)  }\,_{3}F_{2}\left[
\begin{array}
[l]{llll}
2a, & b, & d; & \\
&  &  & 2\\
2b+2, & 1+2a+d-e; &  &
\end{array}
\right] \nonumber\\
&  =\,_{5}F_{4}\left[
\begin{array}
[l]{llllll}
a, & a+\frac{1}{2}, & \frac{1}{2}b+\frac{3}{2}, & \frac{1}{2}d, & \frac{1}{2}d+\frac{1}{2}; & \\
&  &  &  &  & 1\\
\frac{1}{2}b+\frac{1}{2}, & b+\frac{3}{2}, & \frac{1}{2}e, & \frac{1}%
{2}e+\frac{1}{2}; &  &
\end{array}
\right] \nonumber\\
&  +\frac{2ad}{e(b+1)}\,_{4}F_{3}\left[
\begin{array}
[l]{lllll}
a+\frac{1}{2}, & a+1, & \frac{1}{2}d+\frac{1}{2}, & \frac{1}{2}d+1; & \\
&  &  &  & 1\\
b+\frac{3}{2}, & \frac{1}{2}e+\frac{1}{2}, & \frac{1}{2}e+1; &  &
\end{array}
\right]  \label{3.4}%
\end{align}

\begin{proof}
Setting $j=2$ in (\ref{2.1}) and simplifying the resulting identity, we are led to the formula (\ref{3.4}).
\end{proof}
\end{corollary}

\begin{corollary}%
\begin{align}
&  \frac{\Gamma\left(e\right)\Gamma\left(e-2a-d\right)}{\Gamma\left(e-2a\right)\Gamma\left(  e-d\right)  }\,_{3}F_{2}\left[
\begin{array}
[l]{llll}
2a, & b, & d; & \\
&  &  & 2\\
2b-2, & 1+2a+d-e; &  &
\end{array}
\right] \nonumber\\
&  =\,_{5}F_{4}\left[
\begin{array}
[l]{llllll}
a, & a+\frac{1}{2}, & \frac{1}{2}b+\frac{1}{2}, & \frac{1}{2}d, & \frac{1}{2}d+\frac{1}{2}; & \\
&  &  &  &  & 1\\
\frac{1}{2}b-\frac{1}{2}, & b-\frac{1}{2}, & \frac{1}{2}e, & \frac{1}%
{2}e+\frac{1}{2}; &  &
\end{array}
\right] \nonumber\\
&  -\frac{2ad}{e(b-1)}\,_{4}F_{3}\left[
\begin{array}
[l]{lllll}
a+\frac{1}{2}, & a+1, & \frac{1}{2}d+\frac{1}{2}, & \frac{1}{2}d+1; & \\
&  &  &  & 1\\
b-\frac{1}{2}, & \frac{1}{2}e+\frac{1}{2}, & \frac{1}{2}e+1; &  &
\end{array}
\right]  \label{3.5}%
\end{align}

\begin{proof}
Setting $j=-2$ in (\ref{2.1}) and simplifying the resulting identity, we are led to the formula (\ref{3.5}).
\end{proof}
\end{corollary}

\begin{corollary}%
\begin{align}
&  \frac{\Gamma\left(  e\right)  \Gamma\left(  e-2a-d\right)  }{\Gamma\left(
e-2a\right)  \Gamma\left(  e-d\right)  }\,_{3}F_{2}\left[
\begin{array}
[l]{llll}
2a, & b, & d; & \\
&  &  & 2\\
2b+3, & 1+2a+d-e; &  &
\end{array}
\right] \nonumber\\
&  =\,_{5}F_{4}\left[
\begin{array}
[l]{llllll}
a, & a+\frac{1}{2}, & \frac{1}{4}b+\frac{3}{2}, & \frac{1}{2}d, & \frac{1}%
{2}d+\frac{1}{2}; & \\
&  &  &  &  & 1\\
\frac{1}{4}b+\frac{1}{2}, & b+\frac{3}{2}, & \frac{1}{2}e, & \frac{1}%
{2}e+\frac{1}{2}; &  &
\end{array}
\right] \nonumber\\
&  +\frac{6ad}{e(2b+3)}\,_{5}F_{4}\left[
\begin{array}
[l]{llllll}
a+\frac{1}{2}, & a+1, & \frac{3}{4}b+\frac{5}{2}, & \frac{1}{2}d+\frac{1}{2}, & \frac{1}{2}d+1; & \\
&  &  &  &  & 1\\
\frac{3}{4}b+\frac{3}{2}, & b+\frac{5}{2}, & \frac{1}{2}e+\frac{1}{2}, &
\frac{1}{2}e+1; &  &
\end{array}
\right]  \label{3.6}%
\end{align}

\begin{proof}
Setting $j=3$ in (\ref{2.1}) and simplifying the resulting identity, we are
led to the formula (\ref{3.6}).
\end{proof}
\end{corollary}

\begin{corollary}%
\begin{align}
&  \frac{\Gamma\left(  e\right)  \Gamma\left(  e-2a-d\right)  }{\Gamma\left(
e-2a\right)  \Gamma\left(  e-d\right)  }\,_{3}F_{2}\left[
\begin{array}
[l]{llll}
2a, & b, & d; & \\
&  &  & 2\\
2b-3, & 1+2a+d-e; &  &
\end{array}
\right] \nonumber\\
&  =\,_{5}F_{4}\left[
\begin{array}
[l]{llllll}
a, & a+\frac{1}{2}, & \frac{1}{4}b+\frac{3}{4}, & \frac{1}{2}d, & \frac{1}{2}d+\frac{1}{2}; & \\
&  &  &  &  & 1\\
\frac{1}{4}b-\frac{1}{4}, & b-\frac{3}{2}, & \frac{1}{2}e, & \frac{1}{2}e+\frac{1}{2}; &  &
\end{array}
\right] \nonumber\\
&  -\frac{6ad}{e(2b-3)}\,_{5}F_{4}\left[
\begin{array}
[l]{llllll}
a+1, & a+\frac{1}{2}, & \frac{3}{4}b+\frac{1}{4}, & \frac{1}{2}d+\frac{1}{2}, & \frac{1}{2}d+1; & \\
&  &  &  &  & 1\\
\frac{3}{4}b-\frac{3}{4}, & b-\frac{1}{2}, & \frac{1}{2}e+\frac{1}{2}, &
\frac{1}{2}e+1; &  &
\end{array}
\right]  \label{3.7}%
\end{align}

\begin{proof}
Setting $j=-3$ in (\ref{2.1}) and simplifying the resulting identity, we are led to the formula (\ref{3.7}).
\end{proof}
\end{corollary}

\bigskip The results (\ref{3.1}) is the well known result of Krattenthaler and
Rao \cite{5} and the results (\ref{3.2}) to (\ref{3.7}) are closely related to it.

\begin{remark}
We conclude this section by mentioning that the series $_{5}F_{4}$ appearing
on the right-hand side of the results (\ref{3.4}) to (\ref{3.7}) can also be
written as a sum of two $_{4}F_{3}$ and then we can obtain alternate forms of
the results.
\end{remark}

\subsection*{Conflict of Interests}
The authors declare that they have no any conflict of interests.

\subsection*{Acknowledgments}
\begin{enumerate}
\item The work of this research paper was supported by the research grant (05/4/33) funded by Jazan University - Jazan, Saudi Arabia.
\item All authors contributed equally in this paper. They read and approved the final manuscript.
\end{enumerate}


\begin{thebibliography}{9}                                                                                                %


\bibitem {1}W. N. Bailey, Products of generalized hypergeometric series, 
Proc. London Math. Soc., 28 (2), 242 - 254 (1928).

\bibitem {2}W. N. Bailey, Generalized Hypergeometric Series, Cambridge Tracts
in Mathematics and Mathematical Physics, Vol. 32, Cambridge University Press,
Cambridge, London and New York (1935); Reprinted by Stechert-Hafner Service
Agency, New York and London, (1964).

\bibitem {3}B. C. Berndt, Ramanujan's Notebooks, Part II, Springer-Verlage,
Berlin, Heidelberg and New York, (1989).

\bibitem {4}Y. S. Kim, M. A. Rakha and A. K. Rathie,Generalizations of
Kummer's second theorem with applications, Comput. Math \& Math. Phys.,50 (3),
387 - 402, (2010).

\bibitem {5}C. Krattenthaler and K. S. Rao, Automatic generation of
hypergeometric identities by the beta integral method, J. of Comput. Appl.
Math., 160, 159 - 173 (2003).

\bibitem {6}E. E. Kummer, Uber die hypergeometrische Reihe%
\[
1+\frac{\alpha\cdot\beta}{1\cdot\gamma}x+\frac{\alpha(\alpha+1)\cdot
\beta(\beta+1)}{1\cdot2\cdot\gamma(\gamma+1)}x^{2}+\ldots
\]
J. Reine Angew. Math. 15 (1836), pp. 39 - 83 and 127 - 172; see also,
Collected papers, Vol. II: Function Theory, Geometry and Miscellaneous (Edited
and with a forward by And\'{r}e Weil). Springer - Verlag, Berlin, heidelberg
and New York, (1975).

\bibitem {7}E. D. Rainville, Special Functions. The Macmillan Company, New
York (1960); Reprinted by Chelsea Publishing, Bronx, New York, (1971).s, New
York, Chichester (1985).
\end{thebibliography}
\end{document}